\documentclass[final]{amsart}

\RequirePackage[utf8]{inputenc}
\RequirePackage[leqno]{mathtools}
\RequirePackage{amsthm}
\RequirePackage{amssymb}
\RequirePackage[mathcal]{euler}
\RequirePackage{dsfont}

\newtheorem{thm}{Theorem}

\newtheorem{lem}{Lemma}
\newtheorem{cor}{Corollary}

\def\C{\mathds{C} }
\def\tendsto{\longrightarrow}
\def\ind{\mathrm{ind}}
\def\defem#1{{\sl#1}}

\begin{document}

	\begin{abstract}
	Given a convergent sequence of nodes we present a one-dimensional-holomorphic-function
	version of the Newton interpolation method of polynomials.
	It also generalises the Taylor and the Laurent formula. In other words,
	we present an effective identity theorem for complex domains.	\end{abstract}

\title{A note on interpolation series in the complex domain}
\author[T. Sobieszek]{Tomasz Sobieszek}
\email{sobieszek@math.uni.lodz.pl}
\urladdr{http://sobieszek.co.cc}
\address{
		Faculty of Mathematics and Computer Science\\
		University of Łódź\\
		ul. Banacha 22, 90-238 Łódź \\
		Poland
}


\maketitle

\section {Introduction}

We all know well that there is exactly one polynomial of degree at most $n$ that takes given values at
a given set of $n+1$ distinct points $c_1, \ldots, c_{n+1}$. Moreover, this polynomial is given by the following Newton interpolation formula, (see~\cite{Davis},  Sec.~2.6)
\[
	P(x) = P(c_1) + \Delta^1 P(c_1,c_2)(z-c_1) + \ldots +\Delta^n P(c_1, \ldots,c_{n+1}) (z-c_1)\cdots (z-c_n).
\]
The so-called \defem{devided differences} $\Delta^n$ are given by 
\[
	\Delta^1 P(c_1,c_2)= \tfrac{P(c_2)-P(c_1)}{c_2-c_1} , \quad\Delta^1 P(c_1,c_2,c_3)= \tfrac{P(c_1,c_3)-P(c_1,c_2)}{c_3-c_2},
\]
and so on, or equivalently
\[
\Delta ^{n-1} P(c_1,\ldots ,c_n) =
		{\sum _{1\le k\le n} {P(c_k) \over (c_k-c_1)\ldots \widehat{(c_k-c_k)}\ldots (c_k-c_n)}}.
\]

An analogical infinite series for a given infinit set of points $c_n$ is known as the \defem{interpolation series}, (see~\cite{Norlund}).

We generalise the polynomial interpolation formula to give an interpolation series for a holomorphic function, thus establishing an effective identity theorem. In our approach the points $c_i$ don't need to be distinct. In this way we also obtain a generalisation of Taylor or Laurant formula for series expansion of a holomorphic function. 

Two more remarks are due. This result is probably known, but the author is not aware of it. The proofs are simple but somehow neat.

\section {The result}

		\begin{thm}	\label{thm:main}
	Consider two open discs $D_1 ,D_2 \subset \C $ such that $\overline{D_1 } \subset  D_2 $ with centres $c$ and $d$,
	the curves $\Gamma _1 $ and $\Gamma _2 $ which go once around the respective discs in
	the positive direction and a holomorphic function $f: G \tendsto  \C $, where
	$G \supset  \overline{D_2 }\setminus D_1 $.
	Consider also sequences of complex numbers $(c_n)$ and $(d_n)$ convergent
	respectively to $c$ and $d$ and omitting $|\Gamma _2 |$. Then the function $f$ has
	the following representation as the sum of locally uniformly absolutely
	convergent series:
	\[
		f(z) = \sum _{n\ge 1} a_{-n} \, { 1 \over (z-d_1) \ldots  (z-d_n)} 
		\; + \;  \sum _{n \ge 0} a_n \, (z-c_1 ) \ldots  (z-c_n),	\quad (z \in  D_2 \setminus \overline {D_1 })
	\]
	where
	\[
		a_{-n} = {1 \over 2\pi i}\int _{\Gamma _1 } f(\xi ) (\xi -d_1 ) \ldots  (\xi -d_{n-1})  \, d\xi 	\qquad (n\ge 1)
	\]
	and
	\[
		a_n = {1 \over 2\pi i}\int _{\Gamma _2 } {f(\xi ) \over (\xi -c_1 ) \ldots  (\xi -c_{n+1}) } \, d\xi .	\qquad (n\ge 0)
	\]
	\end{thm}

Before we embark on the proof let us first consider the following

		\begin{lem}
	Given a sequence $(\lambda _n)$  of complex numbers and a nonzero complex $x$
	the equality
	\[
		{1 \over x} = \sum _{n\ge 0}  {(\lambda _1  - x)\ldots (\lambda _n - x)	\over
					\lambda _1 \ldots \lambda _n \, \lambda _{n+1}}
	\]
	is satisfied if and only if
	\[
		\prod _{n\ge 1} {\lambda _n - x	\over
				\lambda _n} = 0.
	\]
	This is true for instance for any sequence $(\lambda _n)$  of complex numbers
	such that $| {\lambda _n - x	\over \lambda _n} | < \theta  < 1$ for sufficiently large $n$-s, and then
	the above series is absolutely convergent. \end{lem}
	
		\begin{proof}
	We even have 
	\[
		{1 \over x} = \sum _{n\ge 0}  {(\lambda _1  - x)\ldots (\lambda _n - x)	\over
					\lambda _1 \ldots \lambda _n \, \lambda _{n+1}} + 
				{1 \over x} \prod _{n\ge 1} {\lambda _n - x	\over \lambda _n} ,
	\]
	provided that the sum or the product converges.

	Indeed, this follows from the following identity
	\[
		\sum _{0\le n\le k}  {(\lambda _1  - x)\ldots (\lambda _n - x)	\over
					\lambda _1 \ldots \lambda _n \, \lambda _{n+1}}
		= {1 \over x} - \left( \prod _{1\le n\le k+1} {\lambda _n - x	\over
				\lambda _n} \right) {1 \over x}.
	\]
	To arrive at this equality, observe that the left-hand side is equal to
	\[
		{1 \over \lambda _1 } + {\lambda _1  -x \over \lambda _1 } \left( {1 \over \lambda _2 } + {\lambda _2  -x \over \lambda _2 } \left(\ldots  \left(
		{1 \over \lambda _k} + {\lambda _k -x \over \lambda _k} \left( {1 \over \lambda _{k+1}} \right)\right) \ldots \right)\right).
	\]
	From this we obtain the right-hand side if we bear in mind that
	\[
		{1 \over \lambda _{k+1}} = {1\over x} - {\lambda _{k+1} -x \over \lambda _{k+1}}{1\over x}
	\]
	and that for $\lambda  \neq 0$ we have
	\[
		\left[\, y \longmapsto {1 \over \lambda } + {\lambda -x \over \lambda }y \,\right]
		\;=\;\left[\, {1 \over x} + h \longmapsto {1 \over x} + {\lambda -x\over \lambda } h \,\right].
	\]

	The second part of the lemma is obvious as $(\lambda _n)$ is bounded away from $0$ by
	at least $|x/2|$.
	\end{proof}

We are ready now  to tackle Theorem 1.
		\begin{proof}
	In essence, we proceed as in the standard proof of Laurant series expandability
	of holomorphic functions in annular domains. We begin with the Cauchy theorem
	for the cycle of curves $-\Gamma _1 +\Gamma _2 $:
	\[
		f(z) = {1 \over 2\pi i} \int _{\Gamma _2 } {f(\xi ) \over (\xi -z)} \, d\xi 
			- {1 \over 2\pi i} \int _{\Gamma _1 } {f(\xi ) \over (\xi -z)} \, d\xi .
		\qquad (z \in  D_2 \setminus \overline {D_1 })
	\]
	Then, by Lemma~1, we express $1/(\xi -z)$ firstly in terms of $\xi -c_n$ in $\Gamma _2 $ 
	\[
		{1\over \xi -z} = \sum _{n\ge 0} {(z-c_1 )\ldots (z-c_n)	\over
					(\xi -c_1 )\ldots (\xi -c_{n+1})}, \qquad	(\xi \in |\Gamma _2 |)
	\]
	and secondly in terms of $\xi -d_n$ in $\Gamma _1 $
	\[
		-{1\over \xi -z} = \sum _{n\ge 0} {(\xi -d_1 )\ldots (\xi -d_n)	\over
					(z-d_1 )\ldots (z-d_{n+1})}. \qquad	(\xi \in |\Gamma _1 |)
	\]
	Clearly, these series are locally uniformly absolutely convergent, so we can exchange integrating
	and summing to obtain the asserted expansion.
	
	\end{proof}

Observe that the representation of a holomorphic function as a series of form
\[
	\sum _{n\ge 1} {a_{-n} \over (z-d_1) \ldots  (z-d_n)} +  \sum _{n \ge 0} a_n \, (z-c_1 ) \ldots  (z-c_n)
\]
is not unique when some $d_n$ or $c_n$ are outside the domain of $z$-s in which we consider
such a representation.

For a holomorphic function  $f: G \tendsto  \C $ and any $n\ge 1$ we can define
a holomorphic function $\Delta ^{n-1} f: G^n \tendsto  \C $ by
\[
	\Delta ^{n-1} f (z_1 ,\ldots ,z_n) = {1 \over 2\pi i}\int _{\Gamma } {f(\xi ) \over (\xi -z_1 ) \ldots  (\xi -z_n) } \, d\xi ,
\]
where $\Gamma $ is any cycle of curves for which $\ind_\Gamma  (z_i) =1 $ and $\ind_\Gamma  (z) = 0$ for $z \in  \C \setminus G$.

		\begin{cor}
	Consider an open disc $D\subset \C $ with centre $c$ and a holomorphic function $f: D \tendsto  \C $.
	Consider also a sequence of complex numbers $(c_n)$ such that $c_n \in  D$ tends
	to $c$. Then the function $f$ has the following representation as the sum of a locally
	uniformly absolutely convergent series:
	\[
		f(z) = \sum _{n \ge 0} \Delta^n  f (c_1 ,\ldots ,c_{n+1}) \, (z-c_1 ) \ldots  (z-c_n).	\qquad (z \in  D)
	\]
	The coefficients are unique. \end{cor}

		\begin{proof}
	One only has to apply the Cauchy theorem to Theorem 1 for a sequence of open discs such that their
	closures have $D$ as their union and each of which contains all of $c_n$-s. The uniqueness follows
	from property~7) below. 	\end{proof}

We revert now to the general domain $G$.
We list several simple properties of the function~$\Delta^n $
	\begin{enumerate}
\item  $f \mapsto \Delta ^n f (c_1 ,\ldots ,c_{n+1})$ is a linear functional.
\item	For any permutation $\sigma $ of $\{1,\ldots ,n\}$ we have
	\[
		\Delta ^{n-1} f(c_{\sigma (1)},\ldots ,c_{\sigma (n)}) =\Delta ^{n-1} f(c_1 ,\ldots ,c_n).
	\]
\item	$\Delta ^n f (c,\ldots ,c) = {f^{(n)}(c) \over n!}$ (in particular $\Delta^0 f = f$).
\item	If $c_1 ,\ldots ,c_n$ are all distinct then
	\begin{align*}
		\Delta ^{n-1} f(c_1,\ldots ,c_n) &=
		{\sum _{1\le k\le n} {f(c_k) \over (c_k-c_1)\ldots \widehat{(c_k-c_k)}\ldots (c_k-c_n)}}\\
		&={\sum _{1\le k\le n} {f(c_k) \over (\,z \mapsto (z-c_1 )\ldots (z-c_n)\,)'(c_k)}}.
	\end{align*}
\item	$\Delta ^n f (c_1 ,\ldots ,c_{n+1}) = \Delta ^k\left[ \Delta ^{n-k} f(c_1 ,\ldots ,c_{n-k},\bullet )\right] (c_{n-k+1},\ldots ,c_{n+1})$.
\item \begin{multline*}
		\Delta ^n f (c_1 ,\ldots ,c_{n+1}) =\\
		=\begin{cases} {\Delta ^{n-1} f (c_1 ,\ldots ,c_{n-1},c_{n+1}) - \Delta ^{n-1} f (c_1 ,\ldots ,c_{n-1},c_n)
				\over c_{n+1}-c_n}, &\text{when } c_n\neq c_{n+1};\\
		\left[ \Delta ^{n-1} f (c_1 ,\ldots ,c_{n-1},\bullet ) \right]'(c_n) ,&\text{when } c_n=c_{n+1}.\end{cases}
	\end{multline*}
\item	In the algebra of functions holomorphic on $G$, the value
$\Delta ^n f (c_1 ,\ldots ,c_{n+1})$ is the leading coefficient of
\[
	f(z)\,\text{ \rm{mod} } (z-c_1 )\ldots (z-c_{n+1}),
\]
that is to say if for a holomorphic function g we have
\[
	f(z) = g(z) (z-c_1)\ldots(z-c_{n+1}) + a_n z^n + \text{lower order terms},
\]
say 
\[
	f(z)=a_0 + a_1(z-c_1) + \cdots + a_n(z-c_1)\ldots(z-c_n) + g(z) (z-c_1)\ldots(z-c_{n+1})
\]
then $\Delta ^n f (c_1 ,\ldots ,c_{n+1})=a_n$.
\end{enumerate}
		\begin{proof}
1,2) need no explanation, 3) is a well-known fact, 4) follows from the residue theorem.

The only not-so-trivial property is 5). In case when $c_i$ are all distinct it follows from 4) by $n-k$ substitutions $s=n-k$, $a_q = c_i-c_{n-k+q}$,
for each $i$  from $\{1,\ldots ,n-k\}$
in the following readily-obtainable equality for distinct nonzero $a_1,\ldots ,a_s$
\[
	{1/a_1 \over (a_2-a_1)\ldots (a_s-a_1)} + \ldots  + {1/a_s \over (a_1-a_s)\ldots (a_{s-1}-a_s)}
	= {1\over a_1\ldots a_s}.
\]
(This equality can be proved in a few ways. A. By the residue theorem and the following equality
\[
	\int _\Gamma  {d\xi  \over \xi (\xi -a_1)\ldots (\xi -a_s)} =0,	\quad\text{for large circles $\Gamma $.}
\]
B. By combining 4) for $f(z)=1/z$, Lemma~1 and the uniqueness of decomposition in Collorary~1, or C. Simply by induction.)
The general case follows by the identity theorem in $G^{n+1}$.

Now, property 6) can be obtained by combining 5) with 4), or 3) depending on case. Finally, one
more use of the residue theorem gives 7).
\end{proof}

Motivated by 5) above, we call $\Delta ^n f (c_1 ,\ldots ,c_{n+1})$ the \defem{$n$-th devided difference} of $f$
at points $c_1 ,\ldots ,c_{n+1}$.

\section {Is this a known result?}
In the few sources on interpolation I had the occasion and opportunity
to browse through I haven't actually recognized this result.  However, because of quite a few similar (yet not quite there) theorems I thought it must be known and haven't tried to publish it. 

Now, after a few years of not having anything to do with complex analysis, hoping that at least the method of proof has some new elements to it,  I am placing it on arXiv with the expectation that
someone more knowledgable on the subject will provide me with some valuable feedback.

There is a vast room for improvement in the statement of the theorem and in all the theory
that might surround it. But there would be no point in going into it if it all had been considered before.


\end{document}